\documentclass[a4paper,11pt]{article}
\usepackage[T1]{fontenc}
\usepackage{hyperref}
\usepackage{float}
\usepackage{amssymb,amscd,amsmath,amsfonts,amsthm}
\usepackage{enumerate}
\usepackage{tabularx,environ,array}
\usepackage{tikz}
\usetikzlibrary{decorations.pathreplacing}

\usepackage[numbers,sort]{natbib}
\usepackage{authblk}

\newtheorem{theorem}{Theorem}

\newtheorem{corollary}[theorem]{Corollary}

\newtheorem{lemma}[theorem]{Lemma}

\newtheorem{problem}[theorem]{Problem}
\newtheorem{question}[theorem]{Question}
\newtheorem{proposition}[theorem]{Proposition}
\newtheorem{remark}[theorem]{Remark}

\renewcommand{\emptyset}{\varnothing}

\title{Monophonic number of Kneser graphs and strongly 2-monophonic graphs}
\date{\today}

\author[1,2]{Bo\v{s}tjan Bre\v{s}ar}
\author[3,4]{Mar\'ia Gracia Cornet}
\author[1,2]{Tanja Dravec}

\affil[1]{\footnotesize Faculty of Natural Sciences and Mathematics, University of Maribor, Slovenia}
\affil[2]{\footnotesize Institute of Mathematics, Physics and Mechanics, Ljubljana, Slovenia}
\affil[3]{\footnotesize Depto. de Matem\'atica, FCEIA, Universidad Nacional de Rosario, Argentina}
\affil[4]{\footnotesize Consejo Nacional de Investigaciones Científicas y Técnicas, Argentina}

\begin{document}

\maketitle

\begin{abstract}
 Given a graph $G$ a set $S\subset V(G)$ is called monophonic if every vertex in $G$ lies on some induced path between two vertices in $S$. The monophonic number, $m(G)$, of $G$, which is the smallest cardinality of a monophonic set in $G$, has been studied from various perspectives. In this paper, we establish $m(K(n,r))$ for all Kneser graphs $K(n,r)$, where $n\ge 2r$. In addition, when $r\ge 3$, we prove an even stronger property, notably that every pair of non-adjacent vertices in $K(n,r)$ forms a monophonic set. We call the graphs satisfying this property strongly $2$-monophonic graphs. We present several (sufficient and necessary) conditions for a graph to be strongly $2$-monophonic, and prove that the Cartesian product of any two strongly $2$-monophonic graphs is also such. Besides non-complete Hamming graphs, we also prove that every Johnson graph is strongly $2$-monophonic, whereas chordal graphs, with the exception of the graphs $K_n-e$, do not enjoy this property. 
 \end{abstract}


\medskip


\section{Introduction}
\label{sec: intro}

Over time, significant progress in graph theory has produced and examined numerous classes of graphs. Notably, many of these are connected to classical combinatorial structures with a fundamental example of the generalized Johnson graphs. Formally, given $n\geq r$, for each $i\in\{0,\dots,r\}$, the \textit{generalized Johnson graph} $J(n,r,i)$ is the graph whose vertices are all the $r$-element subsets of a fixed $n$-element set, and two vertices $A$ and $B$ are adjacent iff $|A\cap B|=i$. The cases $i=0$ and $i=r-1$ are the widely-known classes of the Kneser graph $K(n,r)$ and the classical \textit{Johnson graph} $J(n,r)$, respectively.
The interest for Kneser graphs goes back to 1960s and 1970s when two classical theorems concerning their independence and chromatic number were proved~\cite{erdos1961,lovasz1978}. Many other graph invariants have been investigated in Kneser graphs, which makes them one of the most intensively studied classes of graphs.   

A family $\mathcal{C}$ of subsets of a finite set $X$ is a {\em convexity} on $X$ if $\emptyset ,X \in {\mathcal{C}}$ and $\mathcal{C}$ is closed under intersections~\cite{vel-93}. For an interval function $I: X \times X \to 2^X$ a   set $S$ is {\em convex} if $I(x,y) \subseteq S$ for any $x,y \in S$. In the context of graphs, the most prominent instance of interval function is given by the geodesic interval, $I_{G}(u,v)$, which by definition contains all vertices that lie on a shortest path between the vertices $u$ and $v$ in a graph $G$. The {\em geodetic closure} of $S\subset V(G)$, is defined as $I_G \left[ S \right]=\bigcup_{\{x,y\}\subset S}{I_G(x,y)}$, and a set $S$ of vertices in $G$ is a {\it geodetic set} in $G$ if $I_G \left[ S \right] = V (G)$. The {\em geodetic number}, $g(G)$, of a graph $G$, is the minimum cardinality of a geodetic set in $G$.
Besides the geodesic intervals, which are built from shortest paths, the monophonic intervals arising from induced paths have also been intensively studied; see~\cite{chmusi-05} for a survey on various types of interval functions in discrete structures. Monophonic convexity (known also under the names induced path convexity and minimal path convexity) was considered already by Duchet~\cite{Duchet} and Farber and Jamison~\cite{FJ} in the 1980s. 
In this context, given a graph $G$ and $u,v \in V(G)$, by $J_G(u,v)$ we denote the {\it monophonic interval} between $u$ and $v$, which is the set of all vertices that lie on some induced $u,v$-path in $G$. A set of vertices $S\subseteq V(G)$ is a  {\em monophonic set} in $G$, if for every $u\in V(G)$ there exist $x,y\in S$ such that $u\in J_G(x,y)$. The {\it monophonic number}, $m(G)$, of $G$ is the smallest cardinality of a monophonic set in $G$. The {\em monophonic hull number} of $G$ is the smallest cardinality of a set $S\subset V(G)$ such that no proper subset of $V(G)$ that contains $S$ is monophonically convex. 

Geodetic sets and geodesic convexity were studied from various aspects and in numerous classes of graphs; see three comprehensive surveys~\cite{BKT-chapter,cpz,Pela}. 
It was shown in~\cite{Ataci2002} that the determination of $g(G)$ is NP-hard. Moreover, it was proved in~\cite{Do2010} that it is NP-complete to decide for a given chordal or chordal bipartite graph $G$ and a given 
integer $k$, whether $G$ has a geodetic set of cardinality at most $k$.
Recently, the geodetic number was investigated in Kneser graphs~\cite{BEDO2023127964}. For Kneser graphs of diameter 2, the exact values of both the geodetic number and the geodetic hull number were determined, while for Kneser graphs of larger diameter, upper bounds for these parameters were established. 
Concerning the computational complexity issues, there is an interesting discrepancy: there is a polynomial-time algorithm that computes the monophonic hull number of an arbitrary graph, while the determination of $m(G)$ is NP-hard~\cite{DPS}; see also~\cite{CDS} for a continuation of these studies. Hernando et al.~\cite{HJM} investigated relations between Steiner, geodetic and monophonic hull sets, Paluga and Canoy~\cite{PC} studied the monophonic number under the operations of join and lexicographic product, while in~\cite{pal,sancha} monophonic number and monophonic hull number in Cartesian products of graphs were investigated. Recently, the monophonic numbers and monophonic hull numbers of complementary prisms over arbitrary graphs were studied in~\cite{NCN}, while in~\cite{GS} monophonic numbers of corona products of graphs were investigated. 

In this paper, we fill the gap in the literature by studying the monophonic number in Kneser graphs. We determine the exact values of monophonic numbers in all Kneser graph, and prove an even stronger property for a Kneser graph $K(n,r)$, where $r\ge 3$, that every pair of non-adjacent vertices in $K(n,r)$ is a monophonic set. We find this property quite fascinating and call the graphs that satisfy it strongly $2$-monophonic graphs. We prove that several other classes of graphs (such as Johnson graphs and non-complete Hamming graph) are strongly $2$-monophonic, and that some other classes of graphs (such as chordal graphs) have almost no graph with this property. 

The paper is organized as follows. In the remainder of this section we provide some basic definitions and notation, while Section~\ref{s:Kneser} is devoted to our main results concerning the monophonic sets in Kneser graphs. As mentioned earlier $m(K(n,r))=2$ when $r\ge 3$, while $m(K(n,2))=3$. In the proof we combine a careful selection of (induced) paths with the monotonicity of the function $m$ in Kneser graphs which states that $m(K(n,r))$ is non-increasing with increasing $n$.
Section~\ref{s:strongly} is dedicated to the study of strongly 2-monophonic graphs, which are the graphs in which every pair of non-adjacent vertices is a monophonic set. We present several necessary conditions for a graph to be strongly 2-monophonic and investigate the existence of strongly 2-monophonic graphs in several graph classes. We also prove that the Cartesian product of two connected strongly 2-monophonic graphs, distinct from $P_3$, is again strongly 2-monophonic.

\medskip

{\bf{Definitions and notation}}

\medskip

All graphs considered in this paper are finite and simple.

Let $G=(V(G),E(G))$ be a graph. The order of $G$ is denoted by $n(G)$. For a vertex $v \in V(G)$, $N_G(v)$ denotes the set of all neighbors of $v$ in $G$, while $N_G[v]=N_G(v) \cup \{v\}$. 
A vertex $v$ in $G$ is {\em universal} in $G$ if $N[v]=V(G)$.
The \emph{distance} $d_G(u,v)$ between vertices $u,v\in V(G)$ is the length of a shortest path between $u$ and $v$ in $G.$  
The \emph{diameter} of a graph, $diam(G)$, is defined as  $\displaystyle diam(G)=\max\lbrace d_G(u,v)~|~ u,v\in V(G)\rbrace$.
When the graph $G$ is clear from the context, we may omit the subscripts.
A set $S\subseteq V(G)$ is a \emph{clique} of $G$ if $uv\in E(G)$ for each pair $u,v\in V(G)$. The \emph{clique number} of $G$, denoted by $\omega(G)$, is the maximum cardinality of a clique of $G$.
If $G$ is connected, a vertex $v$ of $G$ is called a \emph{cut vertex} if $G-v$ is disconnected, and a subset $S \subseteq V(G)$ is called a \emph{cut set} if $G-S$ is disconnected.

The \emph{Cartesian product} of two graphs $G$ and $H$, denoted by $G \Box H$, is the graph with vertex set $V(G \Box H) = V(G) \times V(H)$, where two vertices $(g,h)$ and $(g',h')$ are adjacent if and only if either $g = g'$ and $hh' \in E(H)$, or $h = h'$ and $gg' \in E(G)$. Hypercubes $Q_n$ and Hamming graphs $H_{m_1,\ldots ,m_n}$ can be defined recursively in the following way: $Q_1=K_2$ and $Q_n=Q_{n-1}\Box K_2$ $H_{m_1}=K_{m_1}$ and $H_{m_1,\ldots , m_n}=H_{m_1,\ldots ,m_{n-1}}\Box K_{m_n}$ for all $n \geq 2$ and arbitrary positive integers $m_1,\ldots ,m_n$.

\section{Monophonic number of Kneser graphs}\label{s:Kneser}

In this section, we first prove that the function $m(K(n,r))$ is non-increasing with respect to $n$. Then we show that $m(K(n,2))=3$ for any $n\geq 5$ and $m(K(n,r))=2$ for $r \geq 3$ and $n \geq 2r+1$. We also prove that a much stronger condition holds, notably that every pair of non-adjacent vertices forms a monophonic set. This property seems to be very interesting and we study it in more details in the next section. 

\smallskip

First note that $m(K(2r,r))=|V(K(2r,r))|=\binom{2r}{r}$.
Note also that, using the standard notation for vertices  of Kneser graphs, a vertex $u \in V(K(n,r))$ also belongs to $V(K(n+1,r))$, since it is represented as an $r$-subset of $[n] \subseteq [n+1]$.

\begin{theorem}\label{thm:monotonicity}
    For any positive integers $n$ and $r$, where $n \geq 2r$, $$m(K(n,r)) \geq m(K(n+1,r)).$$ Moreover, if $S$ is a monophonic set of $K(n,r)$, then $S$ is a monophonic set of $K(n+1,r)$.
\end{theorem}
\begin{proof}
    Let $S \subseteq V(K(n,r))$ be a monophonic set with $|S|=m(K(n,r))$. Let $u=\{i_1,\ldots ,i_{r-1},n+1\}$ be an arbitrary vertex of $V(K(n+1,r))\setminus V(K(n,r))$ and let $I'=[n]\setminus u.$ If $n=2r$, then $|I'|=r+1$ and $S=V(K(2r,r))$. Denote $I'=\{x_1,\ldots ,x_{r+1}\}$ and $x=\{x_1,\ldots ,x_r\}$, $y=\{x_2,\ldots , x_{r+1}\}$. Since $S=V(K(2r,r))$, it follows that $x,y \in S$. Thus $u \in J_{K(2r+1,r)}(x,y)$ as $x,u,y$ is an induced $x,y$-path containing $u$. Hence we may now assume that $n \geq 2r+1$. 
    
    Assume first that there exists $x_t \in I'$ such that $u'=(u\setminus \{n+1\})\cup \{x_t\}=\{i_1,\ldots , i_{r-1},x_t\} \notin S$. Since $S$ is monophonic set of $K(n,r)$ and $u' \in V(K(n,r))$, there exist  $a,b \in S$ such that $u'\in J_{K(n,r)}(a,b)$. Let $P:a=a_0,a_1,\ldots ,a_{\ell-1},a_\ell=u',a_{\ell+1},\ldots , a_m=b$ be an induced $a,b$-path that contains $u'$. Since $a_{\ell-1}$ and $a_{\ell+1}$ are adjacent to $u'=a_\ell$ it follows that $u'\cap a_{\ell-1}=\emptyset$ and $u' \cap a_{\ell+1}=\emptyset$. Since $n+1 \notin a_{\ell-1} \cup a_{\ell+1}$, it follows that $u\cap a_{\ell-1}=\emptyset$ and $u \cap a_{\ell+1}=\emptyset$. Hence $u$ is in $K(n+1,r)$ adjacent to both $a_{\ell-1}$ and $a_{\ell+1}$. Now, let i be the smallest index from $[\ell-1]$ such that $ua_i \in E(K(n+1,r))$ and let $j$ be the largest index in $[m]$ such that $ua_m \in E(K(n+1,r))$. Note that $i\leq \ell-1$ and $j \geq \ell+1$. Then the path $a=a_0,\ldots , a_i,u',a_j,\ldots ,a_m=b$ is an induced $a,b$-path that contains $u$. Thus $u \in J_{K(n+1,r)}(a,b)$ as desired.

    Assume now that for any $x_t \in I'$ it follows that $\{i_1,\ldots ,i_{r-1},x_t\} \in S$. Let $i,j$ be arbitrary different elements of $I'$ and let $\hat{I}=[n]\setminus (u \cup \{i,j\})$. Since $n \geq 2r+1$, we have $|\hat{I}| \geq r-1$. Let $\{j_1,\ldots , j_{r-1}\}$ be an arbitrary subset of $\hat{I}$ of cardinality $r-1$. Let $a=\{i_1,\ldots ,i_{r-1},i\} \in S$ and $b=\{i_1,\ldots ,i_{r-1},j\} \in S$. Then $a,\{j,j_1,\ldots ,j_{r-1}\},u,\{i,j_1,\ldots j_{r-1}\}, b$ is an induced $a,b$-path that contains $u$, that is, $u \in J_{K(n+1,r)}(a,b)$.
\end{proof}

\begin{theorem}\label{thm:r=2}
    If $n\geq 5$, then $m(K(n,2))=3$.
\end{theorem}
\begin{proof}
    It can be easily checked that $m(K(5,2)) \leq 3$, as $\{\{1,2\},\{1,3\},\{2,3\}\}$ is a monophonic set. Hence by Theorem~\ref{thm:monotonicity}, $m(K(n,2)) \leq 3$ for any $n\geq 6$. It remains to show that there does not exist a monophonic set of cardinality 2.

    Let $n\geq 5$ and suppose that there exists a monophonic set $S=\{u,v\}$ of cardinality 2. Since two adjacent vertices cannot form a monophonic set, it follows that $u \cap v \neq \emptyset$. Without loss of generality we may assume that $u=\{1,2\},v=\{1,3\}$. We will show that $w=\{2,3\} \notin J_{K(n,2)}(u,v)$. For the purpose of contradiction suppose that there exists an induced $u,v$-path $P$ that contains $w$. Since any vertex of the $u,w$-subpath $P'$ of $P$ is non-adjacent to $v$, it contains at least one element from $v=\{1,3\}$. Hence the neighbor of $u$ on $P'$ is $u_1=\{3,j\}$ for some $j \in [n] \setminus \{1,2,3\}$ and the neighbor $u_2$ of $w$ on $P'$ is $\{1,i\}$ for some $i \in [n] \setminus \{1,2,3\}$. Now consider the $w,v$-subpath $P''$ of $P$. Since $u$ is not adjacent to vertices of this subpath, it follows that every vertex of $P''$ contains at least one element from $\{1,2\}$. Moreover the neighbor $u_3$ of $w$ on $P''$ is not adjacent to $u_1$, we deduce that $u_3=\{1,j\}$. Similarly, the neighbor $u_4$ of $v$ in $P''$ contains 2 and since $u_4$ is not adjacent to $u_2$, it follows that $u_4=\{2,i\}$. Since $u_4 \cap u_1\ne\emptyset$, this implies $i=j$, which is a contradiction due to $u_2\ne u_3$.  Hence $\{u,v\}$ is not a monophonic set, the final contradiction.  
\end{proof}

In the rest of this section we show that for $r \geq 3$, monophonic number of Kneser graph $K(n,r)$ is 2. Moreover, we show that any two non-adjacent vertices $x,y$ of $K(n,r)$ form a monophonic set. For $r \geq 3$, we first consider odd graphs. The idea of the proof for odd graphs is the following. For arbitrary $u\in V(K(2r+1,r)) \setminus \{x,y\}$ we will construct induced $x,u$- and $y,u$-paths and concatenate them in an induced $x,y$-path that contains $u$. The idea for the construction of $x,u$- and $y,u$-paths is from~\cite{valencia2005diameter}, where authors constructed two types of paths between two vertices $a,b \in V(K(2r+1,r))$ with $|a \cap b|=t$, one of length $2(r-t)$ and the other of length $2t+1.$ We will use these two types of paths and first show that each of them is an induced path.

We introduce the following notation.
For a set $X=\{x_1,\ldots , x_n\}$ and an integer $k$ with $0\leq k\leq n$, we denote $X_{\leq k}:=\{x_1,\ldots ,x_k\}$ and $X_{\geq k}:=\{x_{n-k+1},x_{n-k+2},\ldots , x_n\}$ if $k\geq 1$, $X_{\leq 0}:=\emptyset$ and $X_{\geq 0}:=\emptyset$.

\begin{lemma}\label{l:ValenciaPabon}~\hskip -6pt {\rm (\cite{valencia2005diameter})}
    Let $a,b \in V(K(2r+1,r))$ be two different vertices. If $|a \cap b|=t$, then $d(a,b)=\min{\{2(r-t),2t+1}\}.$  
 \end{lemma}

\begin{lemma}\label{l:ValenciaPabon-1}
    Let $a,b \in V(K(2r+1,r))$ with $|a \cap b|=t \geq 1$. Denote $C=a\cap b$, $A=a \setminus b =\{a_1,\ldots,a_{r-t}\}$, $B=b \setminus a =\{b_1,\ldots,b_{r-t}\}$, $D=[2r+1]\setminus(a \cup b)$ and for $i \in [r-t]$ let $$x_{2i-1}= A_{\leq i-1} \cup B_{\geq r-t-i} \cup D$$ and for $i \in \{0\}\cup[r-t]$ let $$x_{2i}=B_{\leq i} \cup A_{\geq r-t-i} \cup C.$$ 
    Then $P:a=x_0,x_1,\ldots , x_{2(r-t)}=b$ is an induced $a,b$-path of length $2(r-t)$.
\end{lemma}
\begin{proof}
    It is straightforward to verify that $P$ is a path (see also the proof of Lemma~\ref{l:ValenciaPabon} from~\cite{valencia2005diameter}) and clearly its length is $2(r-t)$. It remains to show that $P$ is an induced path, that is, for any $j \geq i+2$, $x_i \cap x_j \neq \emptyset$. First, it is clear that $a=A \cup C$ has a non-empty intersection with $x_k$ for any $k \geq 2$, since $x_k$ contains at least one element from $A$. Now let $i,j \in 2[r-t]$, $j \geq i+2.$ If $i$ and $j$ are of the same parity, then clearly $x_i \cap x_j \neq \emptyset$, as both sets either contain $D$ (if $i,j$ are odd) or $C$ (if $i,j$ are even). Thus, let first $i$ be even and $j$ odd. Then $i=2\ell$ for some $\ell \geq 1$ and $j=2(\ell+k)-1$ for some $k\geq 2$, as $j\geq i+2$. Then $a_{\ell+1} \in x_i \cap x_j$. Finally, let $i=2\ell-1$ for some $\ell \geq 1$ and $j=2(\ell+k)$ for some $k \geq 1$. Then $b_{\ell+1} \in x_i \cap x_j$.  
\end{proof}

\begin{lemma}\label{l:ValenciaPabon-2}
    Let $a,b \in V(K(2r+1,r))$ with $|a \cap b|=t$. Denote $C=a\cap b$, $A=a \setminus b$, $B=b \setminus A$, $D=[2r+1]\setminus(a \cup b)$, let $D'\subseteq D$ with $|D'|=t$    
    and for $i \in [t]$ let $$x_{2i-1}= D'_{\leq i-1} \cup C_{\geq t-i} \cup A\cup (D\setminus D')$$ and for $i \in \{0\}\cup[t]$ let $$x_{2i}=C_{\leq i} \cup D'_{\geq t-i} \cup B.$$ 
    Then $P:a,x_0,x_1,\ldots , x_{2t}=b$ is an induced $a,b$-path of length $2t+1$. 
\end{lemma}
\begin{proof}
    Again, one can verify (possibly using the proof of Lemma~\ref{l:ValenciaPabon} given in~\cite{valencia2005diameter}) that $P$ is a path and clearly its length is $2t+1$. It remains to show that $P$ is an induced path, that is, for any $j \geq i+2$, $x_i \cap x_j \neq \emptyset$ and that $a \cap x_i = \emptyset$ for any $i \geq 1$. Since $A \subseteq x_i$ for any odd $i$ and $c_1 \in x_i \cap a$ for any even $i$, it follows that the only neighbor of $a$ on $P$ is $x_0$. Now let $i,j \in [2t]\cup \{0\}$, $j \geq i+2.$ If $i$ and $j$ are of the same parity, then clearly $x_i \cap x_j \neq \emptyset$, as both sets either contain $A$ (if $i,j$ are odd) or $B$ (if $i,j$ are even). Moreover, it can be easily checked that if $i$ is odd and $j$ is even, then $x_i$ and $x_j$ share an element from $C$ and in the remaining case, when $i$ is even and $j$ is odd, $x_i$ and $x_j$ share an element from $D'$.
\end{proof}

\begin{theorem}\label{thm:odd-strong}
    If $r\geq 3$, then $m(K(2r+1,r))=2$. 
    Moreover, $\{x,y\}$ is a monophonic set of $K(2r+1,r)$ for every pair of non-adjacent vertices $x$ and $y$.
\end{theorem}
\begin{proof}
    Let $x$ and $y$ be two non-adjacent vertices of $K(2r+1,r)$. We will show that $S=\{x,y\}$ is a monophonic set of $K(2r+1,r)$. Thus, let $u$ be an arbitrary vertex in $V(K(2r+1,r))\setminus S$. We will find an induced $x,y$-path that contains vertex $u$.
    If $u\cap(x \cup y)=\emptyset$, then clearly $x,u,y$ is the shortest (and thus induced) $x,y$-path in $K(2r+1,r)$ that contains $u$, and we are done. So, we may assume without loss of generality that $u\cap x\neq\emptyset$, that is, $x$ and $u$ are not adjacent. 
    
    Let us define $X=x\setminus(y\cup u)$, $Y=y\setminus(x\cup u)$, $U=u\setminus(x\cup y)$, $D=x\cap y\cap u$, $A=(x\cap y)\setminus u$, $B=(y\cap u)\setminus x$, $C=(x\cap u)\setminus y$ and $Z=[2r+1]\setminus(x\cup y\cup u)$ (see Fig.~\ref{fig:venn-xyu}). Together, these sets partition $[2r+1]$. Let $t=|u\cap x|=|C|+|D|$ and $s=|u\cap y|=|B|+|D|$.

    \begin{figure}[ht]
        \centering
        \begin{tikzpicture}            
            \draw (150:1) circle (1.5);
            \draw (30:1) circle (1.5);
            \draw (270:1) circle (1.5);

            \node at (150:1.5) {$X$};
            \node at (30:1.5) {$Y$};
            \node at (270:1.5) {$U$};
            
            \node at (0,0) {$D$};
            
            \node at (90:1) {$A$};
            \node at (330:1) {$B$};
            \node at (210:1) {$C$};
            
            \node at (2,-1.5) {$Z$};

            \node at (150:2.75) {$\boldsymbol{x}$};
            \node at (30:2.75) {$\boldsymbol{y}$};
            \node at (270:2.75) {$\boldsymbol{u}$};
        \end{tikzpicture}
        \caption{Venn diagram on the sets $x$, $y$ and $u$.}
        \label{fig:venn-xyu}
    \end{figure}
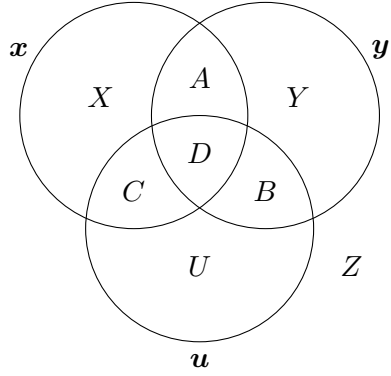
    
    Notice that since $xy \notin E(K(2r+1,r))$ and $xu \notin E(K(2r+1,r))$, we have $|A|+|D|\geq 1$ and $t=|C|+|D|\geq 1$. Besides,
    \begin{equation}\label{eq:cardinality Z}
        |Z|=(2r+1)-(\underbrace{|x|}_{=r}+\underbrace{|Y|+|B|}_{=r-(|x\cap y|)}+\,|U|)=1+\underbrace{|x\cap y|}_{=|A|+|D|}- \, |U|.
    \end{equation}
    Similarly,
    \begin{equation}\label{eq:cardinality Z-2}
        |Z|=1+t-|Y|=1+s-|X|.
    \end{equation}

    We will distinguish different cases.
    \begin{description}
        \item[Case 1.] $y$ and $u$ are adjacent. This is, $y\cap u=\emptyset$. Hence $B=D=\emptyset$.
        
        Let $X'=x\setminus u=X\cup A=\{x_1,\ldots,x_{r-t}\}$, where $x_i\in X$ if and only if $i\leq |X|$.
        
        For $i \in [r-t]$, let
        $$w_{2i-1}=X'_{\leq i-1} \cup U_{\geq r-t-i} \cup (Y\cup Z),$$
        and for $i \in \{0\}\cup[r-t]$, let
        $$w_{2i}=U_{\leq i} \cup X'_{\geq r-t-i} \cup C.$$
        Then considering $a:=x=X'\cup C$, $b:=U\cup C$ and $D:=Y\cup Z$,  Lemma~\ref{l:ValenciaPabon-1}  implies that the path $x=w_0,w_1,\ldots,w_{2(r-t)}=u$ is an induced $x,u$-path.

        We will show that $x=w_0,w_1,\ldots,w_{2(r-t)}=u,y$ is also induced. To do so, we will prove that $y$ is not equal or adjacent to $w_k$ for any $k\in [2(r-t)] \cup \{0\}$. 
        Recall that $y=Y\cup A$, with $Y=y\setminus x\neq\emptyset$ (otherwise $x=y$) and $A=x\cap y\neq\emptyset$ (otherwise $xy \in E(K(2r+1,r))$).
        
        Let $i\in [r-t]$.
        Note that $y\neq w_{2i}$ since $y\setminus w_{2i}\supseteq Y\neq\emptyset$. Further $x_{r-t} \in A \subseteq y$ and $x_{r-t} \notin w_{2i-1}$, since the element $x_{r-t}$ with the largest index in $X'$ is never in $w_{2i-1}$. Hence $y\neq w_{2i-1}$. 
        On the other hand, $y$ is not adjacent to $w_{2i-1}$ since $y\cap w_{2i-1}\supseteq Y\neq\emptyset$. And since $x_{r-t}\in A\cap w_{2i}$ for each $i<r-t$, the only possible neighbor of $y$ is $w_{2(r-t)}=u$ which is adjacent to $y$.

        Therefore, $x=w_0,w_1,\ldots,w_{2(r-t)}=u,y$ is an induced $x,y$-path (of length $2(r-t)+1$) that contains $u$.

        \item[Case 2.] $y\cap u\neq \emptyset$ and $D=Z=\emptyset$. 
        We have $x=X\cup A\cup C$, $y=Y\cup A\cup B$, and $u=U\cup B\cup C$.       
        Notice that from (\ref{eq:cardinality Z}) and (\ref{eq:cardinality Z-2}), since $|D|=|Z|=0$, it follows that $|U|=|A|+1$, $|X|=s+1$ and $|Y|=t+1$. Note that $|C|=t\geq 1$ and $|B|=s\geq 1$ since $u$ intersects $x$ and $y$. Further, $|A|\geq 1$ since $x\cap y\neq \emptyset$.
        Let $B=\{b_1,\ldots,b_s\}$, $X=\{x_1,\ldots,x_{s+1}\}$, $C=\{c_1,\ldots,c_t\}$ and $Y=\{y_1,\ldots,y_{t+1}\}$.

        For $i \in [t]$, let
        $$w_{2i-1}=(Y-\{y_1\})_{\leq i-1} \cup C_{\geq t-i} \cup (X\cup A\cup \{y_1\}),$$ 
        and for $i \in \{0\}\cup[t]$, let
        $$w_{2i}=C_{\leq i} \cup (Y-\{y_1\})_{\geq t-i} \cup (U\cup B).$$
        Considering $a:=x=(X\cup A)\cup C$, $b:=u=(U\cup B)\cup C$, $D:=Y$ and $D':=Y-\{y_1\}$, Lemma~\ref{l:ValenciaPabon-2} implies that the path $P:x,w_0,\ldots,w_{2t}=u$ is an induced $x,u$-path.
        
        For $j \in [s]$, let
        $$v_{2j-1}=(X-\{x_1\})_{\leq j-1} \cup B_{\geq s-j} \cup (Y\cup A\cup \{x_1\}),$$ 
        and for $j \in \{0\}\cup[s]$, let
        $$v_{2j}=B_{\leq j} \cup (X-\{x_1\})_{\geq s-j} \cup (U\cup C).$$
        Considering now $a:=y=(Y\cup A)\cup B$, $b:=u=(U\cup C)\cup B$, $D:=X$ and $D':=X-\{x_1\}$, by Lemma~\ref{l:ValenciaPabon-2} the path $Q:y,v_0,\ldots,v_{2s}=u$ is an induced $y,u$-path.

        We shall show that $x,w_0,w_1,\ldots,w_{2t}=u=v_{2s},v_{2s-1},\ldots,v_0,y$ is an induced $x,y$-path.

        Let $w=w_{2i-1}$ for some $i\in [t]$.
        Note that $x_{s+1}\in w\setminus v_{2j-1}$ for each $j\in [s]$, and $U\subseteq v_{2j}\setminus w$ for each $j\in \{0\}\cup [s]$. So, $w\neq v_k$ for each $k\in \{0\}\cup [2s]$. Besides, $X\subseteq w\setminus y$, so $w\neq y$. Therefore $w$ is not contained in the path $Q$.
        On the other hand, $x_1\in w\cap v_{2j-1}$. If $j<s$, then  $x_{s+1}\in w\cap v_{2j}$. If $j=s$, $v_{2j}=v_{2s}=u$ which is adjacent to $w$ if and only if $i=t$, since $P$ is an induced path. Hence $w$ is not adjacent to any vertex $v$ of $Q$ unless $w=w_{2t-1}$ and $v=u$, but $v$ is the neighbor of $w$  in $P$.

        Now, let $w=w_{2i}$ for some $i\in \{0\}\cup [t-1]$. We have $U\subseteq w\setminus v_{2j-1}$ for all $j\in [s]$. Since $i\leq t-1$, we get $c_t\in v_{2j}$ for each $j\in \{0\}\cup [s]$ but $c_t \notin w$. So, $w\neq v_k$ for each $k\in\{0\}\cup [2s]$, and $w\neq y$ since $U\subseteq w\setminus y$. Hence $w$ is not a vertex in the path $Q$. It remains to show that $w$ is not adjacent to any vertex in $Q$. Clearly, $w$ is not adjacent to $v_{2j}$ for any $j\in \{0\}\cup[s]$ since $U\subseteq w\cap v_{2j}$, and $w$ is not adjacent to $v_{2j-1}$ for any $j\in [s]$ since $y_{t+1}\in w\cap v_{2j-1}$. Finally, $w$ is not adjacent to $y$ as $y_{t+1}\in w\cap y$. 

        Note that $w_{2t}=u$, which is in both paths $P$ and $Q$. Hence $w_{2t}$ has exactly one neighbor in each path, $w_{2t-1}$ and $v_{2s-1}$.

        It follows that for every $k\in\{0\}\cup [2t-1]$, $w_k$ neither belongs to $Q$ nor is adjacent to any of its vertices. Analogously, for every $k\in\{0\}\cup [2s-1]$, $v_k$ neither belongs to $P$ nor is adjacent to any of its vertices. In consequence, $x,w_0,w_1,\ldots,w_{2t}=u=v_{2s},v_{2s-1},\ldots,v_0,y$ is an induced $x,y$-path (of length $2(t+s+1)$) which contains $u$.

        \item[Case 3.] $y\cap u\neq \emptyset$, $D=\emptyset$ and $|Z|=1$. 
        Notice that since $|Z|=1$ and $|D|=0$, from (\ref{eq:cardinality Z}) and (\ref{eq:cardinality Z-2}), it follows that $|X|=|B|=s\geq 1$, $|Y|=|C|=t\geq 1$ and $|U|=|A|\geq 1$. Let $X=\{x_1,\ldots, x_s\}$, $Y=\{y_1,\ldots,y_t\}$, $B=\{b_1,\ldots,b_s\}$ and $C=\{c_1,\ldots,c_t\}$.

        For $i \in [t]$, let
        $$w_{2i-1}=Y_{\leq i-1} \cup C_{\geq t-i} \cup (X\cup A\cup Z),$$ 
        and for $i \in \{0\}\cup[t]$, let
        $$w_{2i}=C_{\leq i} \cup Y_{\geq t-i} \cup (U\cup B).$$
        
        For $j \in [s]$, let
        $$v_{2j-1}=X_{\leq j-1} \cup B_{\geq s-j} \cup (Y\cup A\cup Z),$$ 
        and for $j \in \{0\}\cup[s]$, let
        $$v_{2j}=B_{\leq j} \cup X_{\geq s-j} \cup (U\cup C).$$

        Again by Lemma~\ref{l:ValenciaPabon-2}, the paths $P:x,w_0,w_1,\ldots,w_{2t}=u$ and $Q:y,v_0,v_1,\ldots,v_{2s}=u$ are induced $x,u$-path and $y,u$-path respectively. We aim to show that $x,w_0,w_1,\ldots,w_{2t}=u=v_{2s},v_{2s-1},\ldots,v_0,y$ is an induced $x,y$-path.

        To do so, let us show that for every $k\in\{0\}\cup [2t-1]$, $w_k$ neither belongs to $Q$ nor is adjacent to any of its vertices, and for every $\ell\in\{0\}\cup [2s-1]$, $v_k$ neither belongs to $P$ nor is adjacent to any of its vertices.

        Let $w=w_{2i-1}$ for some $i\in [t]$.
        Note that $y_{t}\in v_{2j-1}\setminus w$ for each $j\in [s]$, and $Z\subseteq w\setminus v_{2j}$ for each $j\in \{0\}\cup [s]$. Thus, $w\neq v_k$ for each $k\in \{0\}\cup [2s]$. Besides, $Z\subseteq w\setminus y$, so $w\neq y$. Therefore $w$ is not contained in the path $Q$. Further, let us see that $w$ is not adjacent to any vertex in $Q$ different from $u$. We have $A\subseteq w\cap v_{2j-1}$, and if $j<s$ (i.e., $v_{2j}\neq u$), then $x_{s}\in w\cap v_{2j}$. If $j=s$, $v_{2j}=v_{2s}=u$ which is adjacent to $w$ if and only if $i=t$, since $P$ is an induced path. Therefore, if $i<t$, $w_{2i-1}$ is not adjacent to any vertex of $Q$, and the only neighbor of $w_{2t-1}$ in $Q$ is $u$, which is its only neighbor in $P$.

        Let $w=w_{2i}$ for some $i\in \{0\}\cup [t-1]$. It holds $Z\subseteq v_{2j-1}\setminus w$ for all $j\in [s]$. And for $i\leq t-1$, $y_t\in w\setminus v_{2j}$. So, $w\neq v_k$ for each $k\in\{0\}\cup [2s]$, and $w\neq y$ since $A\subseteq y\setminus w$. Hence $w$ is not a vertex in the path $Q$. 
        We shall see that $w$ is not adjacent to any vertex in $Q$. Note that $y_t\in w\cap v_{2j-1}$ for every $j\in\cup[s]$. Moreover, $U\subseteq w\cap v_{2j}$ for each $j\in \{0\}\cup [s]$. Finally, $w$ is not adjacent to $y$ as $B\subseteq w\cap y$. Therefore, $w$ is not adjacent to any vertex of $Q$.  

        At last, note that $w_{2t}=u$ is present in both paths $P$ and $Q$, and $u$ has exactly one neighbor in each path, $w_{2t-1}$ and $v_{2s-1}$.

        It follows that $x,w_0,w_1,\ldots,w_{2t}=u=v_{2s},v_{2s-1},\ldots,v_0,y$ is an induced $x,y$-path (of length $2(t+s+1)$) which contains $u$.

        \item[Case 4.] $y\cap u\neq \emptyset$ and $X=Y=\emptyset$. 
        We have $x=A\cup C\cup D$, $y=A\cup B\cup D$, and $u=U\cup B\cup C\cup D$. Since $|x|=|y|=r$, it follows that $|B|=|C|=r-|x\cap y|$ and $s=t$.
        Notice that $|Z|=(2r+1)-(|u|+|A|)=r+1-|A|$. 
        Besides, $A\neq\emptyset$ since otherwise $x=C\cup D=u\cap x$, and as $|u|=r=|x|$, we would have $u=x$, which does not hold. Since $|A|=r-|x \cap u|=r-t$, we get $|Z|=t+1 \geq 2$. Let $Z=\{z_1,z_2,\ldots, z_{t+1}\}$.

        Let $B'=y \cap u=B\cup D=\{b_1,\ldots,b_t\}$, where $b_i\in D$ if and only if $i\geq |B|+1$, and let $C'=x \cap u=C\cup D=\{c_1,\ldots,c_t\}$, with $c_i=b_i$ if and only if $i\geq |C|+1=|B|+1$. 

        For $i \in [t]$, let
        $$w_{2i-1}=(Z-\{z_t\})_{\leq i-1} \cup C'_{\geq t-i} \cup (A\cup \{z_t\}),$$ 
        and for $i \in \{0\}\cup[t]$, let
        $$w_{2i}=C'_{\leq i} \cup (Z-\{z_t\})_{\geq t-i} \cup (B\cup U).$$

        For $j \in [t]$, let
        $$v_{2j-1}=(Z-\{z_{t+1}\})_{\leq j-1} \cup B'_{\geq t-j} \cup (A\cup \{z_{t+1}\}),$$ 
        and for $j \in \{0\}\cup[t]$, let
        $$v_{2j}=B'_{\leq j} \cup (Z-\{z_{t+1}\})_{\geq t-j} \cup (C\cup U).$$

        We claim that $x,w_0,w_1,\ldots,w_{2t}=u=v_{2t},v_{2t-1},\ldots,v_0,y$ is an induced $x,y$-path.

        Again, using Lemma~\ref{l:ValenciaPabon-2}, we can deduce that the path $P:x,w_0,\ldots,w_{2t}=u$ is an induced $x,u$-path and the path $Q:y,v_0,\ldots,v_{2t}=u$ is an induced $y,u$-path.

        Let $w=w_{2i-1}$ for some $i\in [t]$.
        Note that $z_{t}\in w\setminus v_{2j-1}$ for each $j\in [t]$, and $A \subseteq w$ but $A \cap v_{2j} = \emptyset$ for each $j\in \{0\}\cup [t]$. So, $w\neq v_k$ for each $k\in \{0\}\cup [2t]$. Besides, $z_t \in w\setminus y$, so $w\neq y$. Therefore $w$ is not contained in the path $Q$.
        On the other hand, $A\subseteq w\cap v_{2j-1}$ for any $j \in [t]$. If $j<t$, then  $z_t\in w\cap v_{2j}$. If $j=t$, then $v_{2t}=u$ that is clearly adjacent to $w$ only when $i=t$.

        Now, let $w=w_{2i}$ for some $i\in \{0\}\cup [t-1]$. We have $A\subseteq v_{2j-1}$ for all $j\in [t]$, but $A \cap w = \emptyset$ and consequently $w \neq v_{2j-1}$. In addition, as $i\leq t-1$, $z_{t+1}\in w \setminus v_{2j}$ for each $j\in \{0\}\cup [t]$, and thus $w\neq v_{2j}$. Moreover, $w\neq y$ since $A \subseteq y$ and $w \cap A=\emptyset$. Hence $w$ is not a vertex in the path $Q$. It remains to show that $w$ is not adjacent to any vertex in $Q$. Clearly, $w$ is not adjacent to $v_{2j-1}$ for any $j\in [t]$ since $z_{t+1}\in w\cap v_{2j-1}$. It is also clear that $w$ is not adjacent to $y$ as $B \subseteq w\cap y$. On the other hand, $w$ is not adjacent to $v_{2j}$ for any $j\in [t]$ since $b_1 \in w\cap v_{2j}$. Note that $b_1 \in y \cap w$, and thus $w$ is not adjacent to $y$. Now let $j=0$. Then $z_{t-1} \in v_0 \cap w$ unless $i=t-1$ or equivalently $w=w_{2(t-1)}$. In this case $c_1 \in w \cap v_0$.         

        Note that $w_{2t}=u$, which lies in both paths $P$ and $Q$. Hence, $w_{2t}$ has exactly one neighbor in each path, $w_{2t-1}$ and $v_{2t-1}$. 

        Therefore, $x,w_0,w_1,\ldots,w_{2t}=u=v_{2s},v_{2s-1},\ldots,v_0,y$ is an induced $x,y$-path (of length $2(t+s+1)$) which contains $u$.

        \item[Case 5.] $y\cap u\neq \emptyset$, $|X|+|Y|\geq 1$, $|D|+|Z|\geq 1$ and $(|D|,|Z|)\neq (0,1)$.
        
        Let $X'=x\setminus u=A\cup X=\{x_1,\ldots,x_{r-t}\}$, $Y'=y\setminus u=A\cup Y=\{y_1,\ldots,y_{r-s}\}$, with $y_i=x_i$ for each $i\leq |A|$, and let $B'=u\setminus x=U\cup B=\{b_1,\ldots,b_{r-t}\}$, $C'=u\setminus y=U\cup C=\{c_1,\ldots,c_{r-s}\}$, with $c_{i}=b_{i}$ for each $i\leq |U|$.

        For $i \in [r-t]$, let
        $$w_{2i-1}=X'_{\leq i-1} \cup B'_{\geq r-t-i} \cup (Y\cup Z),$$
        and for $i \in \{0\}\cup[r-t]$, let
        $$w_{2i}=B'_{\leq i} \cup X'_{\geq r-t-i} \cup (C\cup D).$$
        Taking $a:=x=(X\cup A)\cup (C\cup D)$, $b:=u=(U\cup B)\cup (C\cup D)$ and $D:=Y\cup Z$, it follows from Lemma~\ref{l:ValenciaPabon-1} that the path $P:x=w_0,\ldots,w_{2(r-t)}=u$ is an induced $x,u$-path. The path is described on Fig.~\ref{fig:Kneser-Pcase5}.

        For $j \in [r-s]$, let
        $$v_{2j-1}=Y'_{\leq j-1} \cup C'_{\geq r-s-j} \cup (X\cup Z),$$
        and for $j \in \{0\}\cup[r-s]$, let
        $$v_{2j}=C'_{\leq j} \cup Y'_{\geq r-s-j} \cup (B\cup D).$$
        If we now consider $a:=y=(Y\cup A)\cup (B\cup D)$, $b:=u=(U\cup C)\cup (B\cup D)$ and $D:=X\cup Z$, by Lemma~\ref{l:ValenciaPabon-1}, the path $Q:y=v_0,\ldots,v_{2(r-s)}=u$ is an induced $y,u$-path. The path is described on Fig.~\ref{fig:Kneser-Qcase5}.

\begin{figure}[ht]
    \centering
    \begin{tikzpicture}    
    \begin{scope}
        \draw (2,-0.3) -- (2,0.3);
        \draw (6,-0.3) -- (6,0.3);
        \draw (7.5,-0.3) -- (7.5,0.3);
        \draw (-2,-0.3) rectangle (9,0.3);

        \node[left] at (-2.25,0) {$x=w_0:$};
        
        \node at (0,0) {$\{x_1,\ldots,x_{r-t-1},x_{r-t}\}$};
        \node at (4,0) {$\emptyset$};
        \node at (6.75,0) {$\emptyset$};
        \node at (8.25,0) {$C\cup D$};
    \end{scope}
    
    \begin{scope}[shift={(0,-1)}]
        \draw (2,-0.3) -- (2,0.3);
        \draw (6,-0.3) -- (6,0.3);
        \draw (7.5,-0.3) -- (7.5,0.3);
        \draw (-2,-0.3) rectangle (9,0.3);

        \node[left] at (-2.25,0) {$w_1:$};
        
        \node at (0,0) {$\emptyset$};
        \node at (4,0) {$\{b_2,b_3,\ldots,b_{r-t}\}$};
        \node at (6.75,0) {$Y\cup Z$};
        \node at (8.25,0) {$\emptyset$};
    \end{scope}
    
    \begin{scope}[shift={(0,-2)}]
        \draw (2,-0.3) -- (2,0.3);
        \draw (6,-0.3) -- (6,0.3);
        \draw (7.5,-0.3) -- (7.5,0.3);
        \draw (-2,-0.3) rectangle (9,0.3);

        \node[left] at (-2.25,0) {$w_2:$};
        
        \node at (0,0) {$\{x_2,\ldots,x_{r-t-1},x_{r-t}\}$};
        \node at (4,0) {$\{b_1\}$};
        \node at (6.75,0) {$\emptyset$};
        \node at (8.25,0) {$C\cup D$};
    \end{scope}

    \begin{scope}[shift={(0,-3)}]
        \draw (2,-0.3) -- (2,0.3);
        \draw (6,-0.3) -- (6,0.3);
        \draw (7.5,-0.3) -- (7.5,0.3);
        \draw (-2,-0.3) rectangle (9,0.3);

        \node[left] at (-2.25,0) {$w_3:$};
        
        \node at (0,0) {$\{x_1\}$};
        \node at (4,0) {$\{b_3,\ldots,b_{r-t}\}$};
        \node at (6.75,0) {$Y\cup Z$};
        \node at (8.25,0) {$\emptyset$};
    \end{scope}
    
    \begin{scope}[shift={(0,-5)}]
        \draw (2,-0.3) -- (2,0.3);
        \draw (6,-0.3) -- (6,0.3);
        \draw (7.5,-0.3) -- (7.5,0.3);
        \draw (-2,-0.3) rectangle (9,0.3);

        \node[left] at (-2.25,0) {$w_{2(r-t)-1}:$};
        
        \node at (0,0) {$\{x_1,\ldots,x_{r-t-1}\}$};
        \node at (4,0) {$\emptyset$};
        \node at (6.75,0) {$Y\cup Z$};
        \node at (8.25,0) {$\emptyset$};
    \end{scope}
    
    \begin{scope}[shift={(0,-6)}]
        \draw (2,-0.3) -- (2,0.3);
        \draw (6,-0.3) -- (6,0.3);
        \draw (7.5,-0.3) -- (7.5,0.3);
        \draw (-2,-0.3) rectangle (9,0.3);

        \node[left] at (-2.25,0) {$u=w_{2(r-t)}:$};
        
        \node at (0,0) {$\emptyset$};
        \node at (4,0) {$\{b_1,b_2,b_3,\ldots,b_{r-t}\}$};
        \node at (6.75,0) {$\emptyset$};
        \node at (8.25,0) {$C\cup D$};
    \end{scope}

    \node[rotate=90] at (3.5,-4) {$\ldots$};
        
    \draw[decoration={brace,mirror,raise=5pt},decorate] (9,-6.35) -- node[right=6pt] {$P$} (9,0.35);
        
    \end{tikzpicture}
    \caption{Induced $x,u$-path $P$ of Case 5.}
    \label{fig:Kneser-Pcase5}
\end{figure}

\begin{figure}[ht]
    \centering
    \begin{tikzpicture}
    
    \begin{scope}
        \draw (2,-0.3) -- (2,0.3);
        \draw (6,-0.3) -- (6,0.3);
        \draw (7.5,-0.3) -- (7.5,0.3);
        \draw (-2,-0.3) rectangle (9,0.3);

        \node[left] at (-2.25,0) {$y=z_0:$};
        
        \node at (0,0) {$\{y_1,\ldots,y_{r-s-1},y_{r-s}\}$};
        \node at (4,0) {$\emptyset$};
        \node at (6.75,0) {$\emptyset$};
        \node at (8.25,0) {$B\cup D$};
    \end{scope}
    
    \begin{scope}[shift={(0,-1)}]
        \draw (2,-0.3) -- (2,0.3);
        \draw (6,-0.3) -- (6,0.3);
        \draw (7.5,-0.3) -- (7.5,0.3);
        \draw (-2,-0.3) rectangle (9,0.3);

        \node[left] at (-2.25,0) {$z_1:$};
        
        \node at (0,0) {$\emptyset$};
        \node at (4,0) {$\{c_2,c_3,\ldots,c_{r-s}\}$};
        \node at (6.75,0) {$X\cup Z$};
        \node at (8.25,0) {$\emptyset$};
    \end{scope}
    
    \begin{scope}[shift={(0,-2)}]
        \draw (2,-0.3) -- (2,0.3);
        \draw (6,-0.3) -- (6,0.3);
        \draw (7.5,-0.3) -- (7.5,0.3);
        \draw (-2,-0.3) rectangle (9,0.3);

        \node[left] at (-2.25,0) {$z_2:$};
        
        \node at (0,0) {$\{y_2,\ldots,y_{r-s-1},y_{r-s}\}$};
        \node at (4,0) {$\{c_1\}$};
        \node at (6.75,0) {$\emptyset$};
        \node at (8.25,0) {$B\cup D$};
    \end{scope}

    \begin{scope}[shift={(0,-3)}]
        \draw (2,-0.3) -- (2,0.3);
        \draw (6,-0.3) -- (6,0.3);
        \draw (7.5,-0.3) -- (7.5,0.3);
        \draw (-2,-0.3) rectangle (9,0.3);

        \node[left] at (-2.25,0) {$z_3:$};
        
        \node at (0,0) {$\{y_1\}$};
        \node at (4,0) {$\{c_3,\ldots,c_{r-s}\}$};
        \node at (6.75,0) {$X\cup Z$};
        \node at (8.25,0) {$\emptyset$};
    \end{scope}
    
    \begin{scope}[shift={(0,-5)}]
        \draw (2,-0.3) -- (2,0.3);
        \draw (6,-0.3) -- (6,0.3);
        \draw (7.5,-0.3) -- (7.5,0.3);
        \draw (-2,-0.3) rectangle (9,0.3);

        \node[left] at (-2.25,0) {$z_{2(r-s)-1}:$};
        
        \node at (0,0) {$\{y_1,\ldots,y_{r-s-1}\}$};
        \node at (4,0) {$\emptyset$};
        \node at (6.75,0) {$X\cup Z$};
        \node at (8.25,0) {$\emptyset$};
    \end{scope}
    
    \begin{scope}[shift={(0,-6)}]
        \draw (2,-0.3) -- (2,0.3);
        \draw (6,-0.3) -- (6,0.3);
        \draw (7.5,-0.3) -- (7.5,0.3);
        \draw (-2,-0.3) rectangle (9,0.3);

        \node[left] at (-2.25,0) {$u=z_{2(r-s)}:$};
        
        \node at (0,0) {$\emptyset$};
        \node at (4,0) {$\{c_1,c_2,c_3,\ldots,c_{r-s}\}$};
        \node at (6.75,0) {$\emptyset$};
        \node at (8.25,0) {$B\cup D$};
    \end{scope}

    \node[rotate=90] at (3.5,-4) {$\ldots$};
        
    \draw[decoration={brace,mirror,raise=5pt},decorate] (9,-6.35) -- node[right=6pt] {$Q$} (9,0.35);
        
    \end{tikzpicture}
    \caption{Induced $y,u$-path $Q$ of Case 5.}
    \label{fig:Kneser-Qcase5}
\end{figure}
        
        We aim to show that $x=w_0,w_1,\ldots,w_{2(r-t)}=u=v_{2(r-s)},v_{2(r-s)-1},\ldots,v_0=y$ is an induced $x,y$-path.
        It suffices to show that for every $k<2(r-t)$ and $\ell<2(r-s)$, $w_k$ is distinct from $v_\ell$ and not adjacent to it.

        \begin{itemize}
            \item Let $k=2i-1$, for some $i\in [r-t]$ and $\ell=2j-1$, for some $j\in [r-s]$.

                If $X\neq \emptyset$, then $x_{r-t}\in X\subseteq v_\ell$, but $x_{r-t}\notin w_k$, since $w_k\cap X'=\{x_1,\ldots,x_{i-1}\}\subseteq X'-\{x_{r-t}\}$. Thus, $w_k\neq v_\ell$. If $X=\emptyset$, then $Y\neq \emptyset$ and with a similar reasoning, $y_{r-s}\in w_k\setminus v_\ell$. Hence $w_k\neq v_\ell$.
        
                If $Z\neq \emptyset$, then $w_k\cap v_\ell\supseteq Z$, and $w_k$ is not adjacent to $v_\ell$. 
                If $Z=\emptyset$, then $D\neq\emptyset$ as $|Z|+|D|\geq 1$. From (\ref{eq:cardinality Z}), it follows that $|U|=|A|+|D|+1>|A|+1$. 
                
                If $i=1$, $w_k=B'-\{b_1\}\cup Y$. Note that $v_\ell\cap Y=\emptyset$ only if $j-1\leq |A|$. Besides, $v_\ell\cap (U-\{b_1\})=\emptyset$ only if $r-s-j\leq |C|$, i.e., $j\geq r-s-|C|=|U|$. Thus, if $w_k$ and $v_\ell$ are adjacent, then $|U|\leq j\leq |A|+1$, but $|U|>|A|+1$. Thus, $w_k$ and $v_\ell$ are not adjacent. 
                Now let $i > 1$. If $j=1$, from a similar reasoning, it follows that $w_k$ is not adjacent to $v_\ell=v_1$. Finally, if $i>1$ and $j>1$, then $x_1\in w_k\cap v_\ell$ either if $x_1\in A$ or if $x_1\in X$ (i.e., if $A=\emptyset$). In any case, it turns out that $w_k$ and $v_\ell$ are not adjacent.
            
            \item Let $k=2i-1$, for some $i\in [r-t]$ and $\ell=2j$, for some $j\in \{0\}\cup [r-s-1]$.

                Note that $|D|+|Z|\geq 1$. So $D$ and $Z$ are not simultaneously empty. Thus, $w_k\neq v_\ell$ since $D\subseteq v_\ell\setminus w_k$ and $Z\subseteq w_k\setminus v_\ell$.
        
                If $Y\neq \emptyset$, then we have that $y_{r-s}\in v_\ell\cap Y\subseteq v_\ell\cap w_k$. Thus, $w_k$ and $v_\ell$ are not adjacent. 
                Otherwise, note that both $A$ and $B$ must be non-empty since $A\cup D\subsetneq y=A\cup B\cup D$ and $B\cup D\subsetneq y=A\cup B\cup D$. Besides, $X\neq \emptyset$ since $|X|+|Y|\geq 1$. 
                
                Thus, if $i<r-t$, $b_{r-t}\in B\cap w_k\subseteq v_\ell\cap w_k$. 
                If $i=r-t$, then $X'\cap w_k = \{x_1,\ldots,x_{r-t-1}\}\supseteq A$ and $y_{r-s}\in A\cap v_\ell$. It follows that $w_k\cap v_\ell\neq \emptyset$. In either case, $w_k$ is not adjacent to $v_\ell$.
            
            \item Let $k=2i$, for some $i\in \{0\}\cup [r-t-1]$ and $\ell=2j-1$, for some $j\in [r-s]$.

                Analogously to what was done for $w_{2i-1}$ and $v_{2j}$, it follows that $w_{2i}$ and $v_{2j-1}$ are neither equal nor adjacent.
            
            \item Let $k=2i$, for some $i\in \{0\}\cup [r-t-1]$ and $\ell=2j$, for some $j\in \{0\}\cup [r-s-1]$.

                If $Y\neq \emptyset$, then we have that $y_{r-s}\in v_\ell \cap Y$. Thus, $y_{r-s}\in v_\ell\setminus w_k$ since $w_k\cap Y=\emptyset$. If $Y=\emptyset$, then $X\neq\emptyset$ since $|X|+|Y|\geq 1$, and using the same reasoning it follows that $x_{r-t}\in w_k\setminus v_\ell$. In any case, $w_k\neq v_\ell$.

                If $D\neq \emptyset$, then $D\subseteq w_k\cap v_\ell$. Thus, $w_k$ and $v_\ell$ are not adjacent.
                If $D= \emptyset$, then $B$ and $C$ are not empty since $B=y\cap u\neq \emptyset$ and $C=x\cap u\neq \emptyset$. Besides, $|Z|\geq 2$ since $|D|+|Z|\geq 1$ and $(|D|,|Z|)\neq (0,1)$. Hence from (\ref{eq:cardinality Z}) it follows that $|A|\geq |U|+1$.
                
                If $i=0$, then $w_k=x=A\cup X\cup C$. Note that $v_\ell\cap C=\emptyset$ only if $j\leq |U|$. Besides, $v_\ell\cap A=\emptyset$ only if $r-s-j\leq |Y|$, i.e., $j\geq r-s-|Y|=|A|$. Hence if $w_k\cap v_\ell=\emptyset$, then $|A|\leq j\leq |U|$. But we arise to a contradiction since $|A|\geq |U|+1$. Therefore, $w_k$ and $v_\ell$ are not adjacent. 
                Now let $i > 0$. If $j=0$, with a similar reasoning, it follows that $w_k$ and $v_0=y$ are not adjacent. Finally, if $i\geq 1$ and $j\geq 1$, then $b_1\in w_k\cap v_\ell$, either if $b_1\in U$ or if $b_1\in B$ (i.e., $U=\emptyset$). In any case, $w_k$ and $v_\ell$ are not adjacent.
        \end{itemize}
         \end{description}
\end{proof}

The following corollary of Theorem~\ref{thm:odd-strong} is the main result of this section.

\begin{corollary}\label{cor:monophonic=2}
    If $r \geq 3$ and $n \geq 2r+1$, then $m(K(n,r))=2$. Moreover, $\{x,y\}$ is a monophonic set of $K(2r+1,r)$ for every pair of non-adjacent vertices $x$ and $y$.
\end{corollary}
\begin{proof}
    Let $r\geq 3$. Then Theorem~\ref{thm:odd-strong} implies that $m(K(2r+1,r))=2$ and hence for any $n\geq 2r+1$, $m(K(n,r)) \leq 2$, by Theorem~\ref{thm:monotonicity}. Then $m(K(n,r))=2$, as every non-trivial graph $G$ has monophonic number at least 2. 

    Now let $x,y$ be two arbitrary non-adjacent vertices of $K(n,r)$. Hence $|x \cup y| \leq 2r-1$. Without loss of generality we may assume that $x \cup y \subseteq [2r-1]$. Hence $x,y$ are also non-adjacent vertices of the graph $K(2r+1,r)$ and thus by Theorem~\ref{thm:odd-strong} $\{x,y\}$ is a monophonic set of $K(2r+1,r)$. Therefore, it follows from Theorem~\ref{thm:monotonicity} that $\{x,y\}$ is a monophonic set of $K(n,r)$ for any $n\geq 2r+1$. 
\end{proof}

\section{Strongly 2-monophonic graphs}\label{s:strongly} 

In this section, we prove several properties of strongly 2-monophonic graphs and provide some necessary conditions for a graph to be in this class of graphs. Then, we characterize strongly 2-monophonic graphs within several graph families, such as  Johnson graphs and chordal graphs, and prove that the property of being strongly $2$-monophonic is preserved by the Cartesian product operation unless one of the factors is $P_3$. 

A graph $G$ is \emph{$2$-monophonic} if $m(G)=2$. A 2-monophonic graph is \emph{strongly $2$-monophonic} if $\{x,y\}$ is a monophonic set for every pair of non-adjacent vertices $x,y \in V(G)$. It follows directly from the definition that the only strongly 2-monophonic graphs of order at most 3 are $2K_1,K_2$ and $P_3$, and that a strongly $2$-monophonic graph is connected as soon as its order is greater than $2$.
Recall that Corollary~\ref{cor:monophonic=2} yields that every Kneser graph $K(n,r)$, where $r\ge 3$ and $n\ge 2r+1$ is strongly 2-monophonic.

\subsection{Basic properties of strongly 2-monophonic graphs}

In this subsection, we present several nice properties of strongly 2-monophonic graphs, some of which will be used in the sequel.

First, we recall another concept from the theory of graph convexities. A set $S \subseteq V(G)$ is {\em $m$-convex} if $J_G(u,v) \subseteq S$ for any $u,v \in S$. The {\it $m$-convexity number} of $G$, $c_m(G)$, is the cardinality of a maximum proper $m$-convex set $S \subsetneq V(G)$.

\begin{proposition}\label{prop:convexSets}
    If $G$ is a strongly 2-monophonic graph, then the only proper $m$-convex sets of $G$ are cliques. Moreover, $c_m(G)=\omega(G)$. 
\end{proposition}
\begin{proof}
     Suppose that there exists an $m$-convex set $S \subset V(G)$ that is not a clique. Let $a \in V(G) \setminus S$. Since $S$ is not a clique, there exist $x,y \in S$ such that $xy \notin E(G)$. Hence $\{x,y\}$ is a monophonic set of $G$. Thus $a \in J_G(x,y)=V(G)$, which contradicts the fact that $S$ is $m$-convex. Since any clique is $m$-convex, it follows that $c_m(G)=\omega(G)$.
\end{proof}

In the following result, we present several necessary conditions that a strongly $2$-monophonic graph must satisfy.

\begin{proposition}\label{prp:necessary}
If $G$ is a strongly $2$-monophonic graph, then the following properties hold:
\begin{enumerate}[(i)]
\item $G$ has no cut vertices unless $G=P_3$;
\item for any $x\in V(G)$ the set $N[x]$ is not a cut set;
\item if $N(y)\subseteq N(x)$ for some $x,y\in V(G)$, then 
$N[x]=V(G)\setminus\{y\}$;
\item if $N[y]\subseteq N[x]$ for some $x,y\in V(G)$, then $x$ is a universal vertex.
\end{enumerate}
\end{proposition}
\begin{proof}
(i)   Let $v$ be a cut vertex of a strongly $2$-monophonic graph $G$, which is not $P_3$, and $H_1,\ldots,H_k$ the connected components of $G-v$. First, assume $H_i$ has two non-adjacent vertices $x$ and $y$ for some $i\in [k]$, and let $u \in V(H_j)$ for $j\in [k]$, $j\neq i$. Since any $x,u$-path contains $v$ and the same holds for any $y,u$-path, there does not exist any $x,y$-path which contains $u$. So, $H_t$ is a clique for every $t\in [k]$. Suppose that $|V(H_i)| \geq 2$ for some $i \in [k]$. Let $x \in V(H_i)$ be a vertex adjacent to $v$  and let $u \in V(H_i) \setminus \{x\}$ and $y \in V(H_j)$ for $j \neq i$. Then any $x,y$-path $P$ that contains $u$ is not induced as $P$ contains $v$ and the edge $vx$. Hence $\{x,y\}$ is not a monophonic set, a contradiction.
Thus, $|V(H_t)|=1$ for each $t\in [k]$.
Since $G$ is not isomorphic to $P_3$, it follows that $k\geq 3$. Considering $x\in V(H_1)$, $y\in V(H_2)$ and $u\in V(H_3)$, the only induced $x,y$-path is $x,v,y$ which clearly does not contain $u$. Hence $\{x,y\}$ is not a monophonic set, a contradiction.   

(ii) Suppose to the contrary that $N[x]$, where $x\in V(G)$, is a cut set in a strongly $2$-monophonic graph $G$. Let $u$ and $y$ lie in different components of $G-N[x]$. Since $G$ is strongly $2$-monophonic and $xy\notin E(G)$, we infer that $\{x,y\}$ is a monophonic set of $G$. However, every $y,x$-path that contains $u$ passes a vertex in $N(x)$, therefore such a path is not induced, a contradiction.  

(iii) Let $G$ be a strongly $2$-monophonic graph and $x,y\in V(G)$ with $N(y)\subseteq N(x)$. Assume that $N(x)\neq V(G)\setminus\{x,y\}$ and let $v\in V(G) \setminus (N[x]\cup\{y\})$. Since $xv\notin E(G)$, it follows that $\{x,v\}$ is a monophonic set of $G$. Since any neighbor of $y$ is also a neighbor of $x$, there does not exist an induced $x,v$-path that contains $y$, a contradiction. Hence $N[x]=V(G)\setminus \{y\}$.

Finally, (iv) can be proved by following arguments similarly as in item (iii).  
\end{proof}

We end this subsection with the following sufficient condition for strong $2$-monophonicity.

\begin{proposition}
\label{prp:sufficient}
If $G$ is a non-complete graph such that for any three vertices one of which is not adjacent to any of the other two there exists an induced cycle passing through all three vertices, then $G$ is strongly $2$-monophonic. 
\end{proposition}
\begin{proof}
    Let $x,y$ be non-adjacent vertices of $G$, and $u\in V(G)\setminus\{x,y\}$. If $u\in N(x)\cap N(y)$, then clearly $u\in J_G(x,y)$. Thus, without loss of generality assume $ux\notin E(G)$. Since $x$ is not adjacent to either $y$ and $u$, there exists an induced cycle $C$ passing through $x$, $y$ and $u$. The $x,y$-subpath of $C$ containing $u$ is an induced $x,y$-path, so $u\in J_G(x,y)$.
\end{proof}

To see that the condition in Proposition~\ref{prp:sufficient} is not also necessary for a graph to be strongly $2$-monophonic, consider the graph $G$ (isomorphic to $K_3\Box K_2$) depicted on Fig.~\ref{fig:hamming}. The three shaded vertices have the property that one of them is at distance $2$ from the other two, yet the three vertices do not lie on an induced cycle. On the other hand, $G$ is strongly $2$-monophonic. 

\begin{figure}[htb]
\begin{center}
\begin{tikzpicture}[scale=.9,style=thick,x=1cm,y=1cm]
\def\vr{3.25pt}

\path (0,0) coordinate (v1);
\path (2,0) coordinate (v2);
\path (4,0) coordinate (v3);
\path (0,2) coordinate (v4);
\path (2,2) coordinate (v5);
\path (4,2) coordinate (v6);

\draw (v1)--(v2)--(v3)--(v6)--(v5)--(v4);
\draw (v1)--(v4); 
\draw  (v2)--(v5);
\draw (v6) to[out=90,in=90, distance=1cm] (v4);
\draw (v3) to[out=90,in=90, distance=-1cm] (v1);

\draw (v1) [fill=black] circle (\vr);
\draw (v2) [fill=white] circle (\vr);
\draw (v3) [fill=white] circle (\vr);
\draw (v4) [fill=white] circle (\vr);
\draw (v5) [fill=black] circle (\vr);
\draw (v6) [fill=black] circle (\vr);

\end{tikzpicture}
\caption{A strongly $2$-monophonic graph; the three shaded vertices do not belong to an induced cycle.}
\label{fig:hamming}
\end{center}
\end{figure}
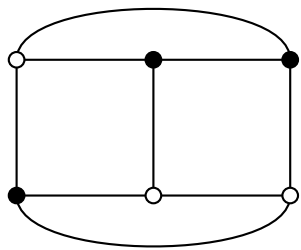

\subsection{Chordal graphs and dismantlable graphs}
\label{ss:chordal}
In this subsection, we study two famous classes of graphs. We prove that with some very special exceptions chordal graphs are not strongly $2$-monophonic. On the other hand, from any strongly $2$-monophonic graph one can create a dismantlable strongly $2$-monophonic graph by adding a universal vertex. 

A graph is \emph{chordal} if it contains no induced cycle of length greater than 3. A vertex $v$ of a graph $G$ is called \emph{simplicial} if the subgraph of $G$ induced by $N[v]$ is a complete graph. Chordal graphs have many interesting properties. It is well-known that every minimal separator in a chordal graph is a clique. Moreover, every chordal graph that is not a complete graph has two non-adjacent simplicial vertices~\cite{dirac1961rigid}.

A vertex $x$ of an $m$-convex set $S$ in a graph $G$ is {\em $m$-extreme} if $S-\{x\}$ is also $m$-convex.  It is well-known that $m$-extreme vertices are exactly simplicial vertices~\cite{FJ}. Since by definition, $m$-extreme vertex $u$ does not lie on $J_G(x,y)$ for any $x,y \in S\setminus \{u\}$, it holds that the set of $m$-extreme vertices of $G$ is contained in any monophonic set $S$. This implies the following.

\begin{remark}\label{rem:2-monophonic}
    If $G$ is 2-monophonic graph, then $G$ has at most two simplicial vertices.
\end{remark}

\begin{proposition}\label{prp:S2M-1}
    If $G$ is a strongly 2-monophonic graph having two non-adjacent simplicial vertices $x,y$, then ${\rm{deg}}(z)=n(G)-1$ for any $z \in V(G)\setminus \{x,y\}$. In particular, $G-\{x,y\}$ is a clique.
\end{proposition}

\begin{proof}
    By Remark~\ref{rem:2-monophonic}, $x$ and $y$ are the only simplicial vertices of $G$. Since $m(G)=2$ and since every simplicial vertex is contained in a monophonic set, it follows that $S=\{x,y\}$ is the only monophonic set of $G$ of cardinality 2. Since $G$ is strongly 2-monophonic, $x$ and $y$ are the only two non-adjacent vertices of $G$. 
\end{proof}

\begin{corollary}\label{c:S2Mchordal}
If $G$ is a chordal graph of order at least 3, then $G$ is a strongly $2$-monophonic graph if and only if $G=K_n-e$, for some $n\ge 3$.    
\end{corollary}
\begin{proof}
    If $G$ is isomorphic to $K_n-e$, for $n\geq 3$, then clearly the two vertices $x$ and $y$ of degree $n-2$ form a monophonic set. Moreover, this is the only pair of two non-adjacent vertices, thus $K_n-e$ is a strongly 2-monophonic graph.

    For the converse let $G$ be a strongly 2-monophonic graph of order at least 3. It is clear that $G$ is not a complete graph, as $m(K_n)=n\geq 3$ for any $n \geq 3$. Since $G$ is chordal and not complete, $G$ has at least two non-adjacent simplicial vertices. Hence Remark~\ref{rem:2-monophonic} implies that $G$ has exactly 2 simplicial vertices $x,y$. By Proposition~\ref{prp:S2M-1} $x$ and $y$ are the only two non-adjacent vertices of $G$. Thus $G$ is isomorphic to $K_{n(G)}-xy$.
\end{proof}

A vertex $u$ is {\em dominated by} a vertex $v$ if $N[u]\subseteq N[v]$.
A graph $G$ of order at least $2$ is {\em dismantlable} if there exists a vertex $u$, which is dominated by some other vertex, and $G-u$ is also dismantlable. This property provides an elimination scheme in dismantlable graphs by which dominated vertices are removed one-by-one from the graph until one ends up with the graph $K_1$. As proved by Nowakowski and Winkler~\cite{NR}, dismantlable graphs are precisely cop-win graphs. 

By Proposition~\ref{prp:necessary}(iv), if $G$ is a strongly $2$-monophonic dismantlable graph, then there exists a universal vertex in $G$. 
On the other hand, if $G$ has universal a vertex $v$, then $G$ is strongly $2$-monophonic if and only if $G-v$ is strongly $2$-monophonic. From these two observation, we derive the following observation.
\begin{proposition}
\label{prp:dismant}
If $G$ is a dismantlable non-complete graph, then $G$ is a strongly $2$-monophonic graph if and only if $G$ has a non-empty set $U$ of universal vertices and $G-U$ is a strongly $2$-monophonic graph without universal vertices.    
\end{proposition}

\subsection{Cartesian products}

We start the section with the following trivial remark, that will be needed in the proof of the main result of this section.

\begin{remark}\label{r:cartesian}
    Let $G$ and $H$ be connected graphs with $|V(G)| \geq 3, |V(H)|\geq 2$. If $\{g_1,g_2\}$ is a monophonic set of $G$, then for an arbitrary $h \in V(H)$, the set $\{(g_1,h),(g_2,h)\}$ is a monophonic set of $G\Box H$. 
\end{remark}

Remark~\ref{r:cartesian} implies that if $G$ is a 2-monophonic graph of order at least 3, then for any connected graph $H$ it follows that $m(G \Box H)=2$, that is, $G \Box H$ is 2-monophonic. The result cannot be generalized to strongly 2-monophonic graph $G$, as there are examples of strongly 2-monophonic graphs $G$ and $H$ such that $G \Box H$ is not strongly 2-monophonic. For example, $P_3$ is strongly 2-monophonic and Proposition~\ref{prp:necessary}(ii) implies that $P_3 \Box P_3$ is not strongly 2-monophonic. Moreover, the graph $P_3\Box K_2$ is not strongly $2$-monophonic.
In the main result of this section we prove that the statement is true for any connected strongly 2-monophonic graph $G$ different from $P_3$. To prove this we first need the following lemma.

\begin{lemma}\label{l:disjointpaths}
    Let $G$ be a strongly $2$-monophonic graph not isomorphic to $P_3$, and $x,y,z$ three distinct vertices of $G$. Then there exist an induced $x,y$-path $P$ and an induced $y,z$-path $Q$ such that the only vertex in both $P$ and $Q$ is $y$. 
\end{lemma}
\begin{proof}
    Let $x,y,z\in V(G)$. If $xz\notin E(G)$, then since $G$ is strongly $2$-monophonic, there exists an induced $x,z$-path $R$ containing $y$. Hence if we denote by $P$ the $x,y$-subpath of $R$ and by $Q$ the $y,z$-subpath of $R$, then  $P$ and $Q$ are induced paths whose only vertex in common is $y$.
    
    Now, assume $xz\in E(G)$. If $y\in N(x)\cap N(z)$, then it is enough to consider $P:x,y$ and $Q:y,z$. 
    Otherwise, either $yz\notin E(G)$ or $yx\notin E(G)$. 
    We may assume without loss of generality that $y$ is not adjacent to $z$. Since $z$ is not universal, we infer by using Proposition~\ref{prp:necessary}(iv) that $N(x)$ is not contained in $N[z]$. Hence, let $x'\in N(x)\setminus N[z]$. 

    If $x'=y$, then by Proposition~\ref{prp:necessary}(i) it follows that $x$ is not a cut vertex of $G$. Thus, there exists an induced $y,z$-path $Q$ which does not contain $x$. Hence $P:x,y$ and $Q$ are induced $x,y$- and $y,z$-paths respectively with no vertices in common other than $y$. Thus, assume $x'\neq y$. Since $x'z\notin E(G)$, there exists an induced $x'z$-path $P':x'=v_1,\ldots,v_k=z$ which contains $y$, that is, where $y=v_j$ for some $j\in [k-1]$. 
    Observe that $x$ is not a vertex in $P'$ since $xx',xz\in E(G)$, but $x$ might be adjacent to some vertices of $P'$. Let $t=\max\{i\leq j:v_ix\in E(G)\}$.
    We may consider the paths $P:x,v_t,v_{t+1},\ldots,v_j=y$ and $Q:y=v_j,\ldots, v_k=z$. 
    It follows that $P$ and $Q$ are induced $x,y$- and $y,z$-paths which only share vertex $y$.
\end{proof}

\begin{theorem}\label{thm:product}
    If $G$ and $H$ are connected strongly $2$-monophonic graphs different from $P_3$, then $G\Box H$ is strongly $2$-monophonic.
\end{theorem}
\begin{proof}
    Let $(x,a),(y,b)$ be non-adjacent vertices of $G\Box H$. We will show that $\{(x,a),(y,b)\}$ is a monophonic set.  That is, for any other vertex $(z,c)$ of $G\Box H$, there exist an induced $(x,a)(y,b)$-path containing $(z,c)$.

    Firstly, assume $a=b$. In this case $(x,a)$ and $(y,a)$ are in the same $G$-layer and $xy\notin E(G)$. Since $G$ is strongly 2-monophonic, $\{x,y\}$ is a monophonic set of $G$ and thus $\{(x,a),(y,a)\}$ is a monophonic set of $G \Box H$ by Remark~\ref{r:cartesian}. 

    Now, consider $a\neq b$. Observe that the case $x=y$ can be resolved in the same way as in the previous paragraph by reversing the roles of $G$ and $H$. So, we may assume $x\neq y$.
    We will distinguish two cases.

    \begin{description}
        \item[Case 1.] If $z\in \{x,y\}$ or $c\in \{a,b\}$. 

        If both conditions hold, we can assume, without loss of generality, that $z=x$. Then, assuming that $c=b$, we then find a shortest $(x,a),(y,b)$-path, which contains $(z,c)=(x,b)$ and we are done. Indeed, letting $x=v_0,\ldots, v_k=y$ be a shortest $x,y$-path in $G$, and $a=d_0,\ldots,d_j=b$ be a shortest $a,b$-path in $H$, the path $(x,a)=(x,d_0),\ldots,(x,d_j)=(x,b)=(v_0,b),\ldots,(v_k,b)=(y,b)$ is a shortest $(x,a),(y,b)$-path containing $(x,b)$.

        Now, let $z\notin\{x,y\}$ and $c\in \{a,b\}$. Assume without loss of generality that $c=a$. Let $Q:a=d_0,\ldots,d_j=b$ be a shortest $a,b$-path in $H$. 
        By Lemma~\ref{l:disjointpaths}, there exist two induced paths in $G$, $P_1:x=v_0,\ldots,v_{k_1}=z$ and $P_2:z=u_0,\ldots,u_{k_2}=y$ that only share vertex $z$. Thus, consider the following path in $G\Box H$: $P:(x,a)=(v_0,a),\ldots,(v_{k_1},a)=(z,a)=(z,d_0),\ldots,(z,d_j)=(z,b)=(u_0,b),\ldots,(u_{k_2},b)=(y,b)$, which is a $(x,a),(y,b)$-path containing $(z,a)$, and is induced since $P_1$ and $P_2$ do not share any vertex except from $z$.
        Note that the only vertex of this path with first coordinate $x$ is vertex $(x,a)$. The path $P$ is depicted on Fig.~\ref{fig:product-case1}.
        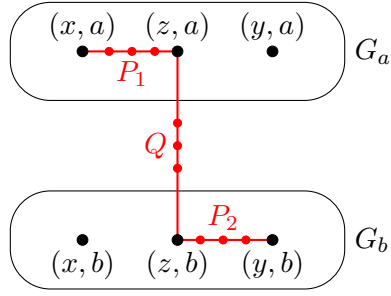
\begin{figure}[ht]
        \centering
        \begin{tikzpicture}[rotate=90]
            \draw[rounded corners=15pt] (-0.65,-2.2) rectangle (0.65,2.2);
            \node[right] at (0,-2.2) {$G_b$};
            \draw[rounded corners=15pt] (1.85,-2.2) rectangle (3.15,2.2);
            \node[right] at (2.5,-2.2) {$G_a$};
    
            \draw[thick, red] (0,-1.25) -- (0,0) -- (2.5,0) -- (2.5,1.25);
    
            \filldraw[red] (0,-0.9) circle (1.5pt);
            \filldraw[red] (0,-0.6) circle (1.5pt);
            \filldraw[red] (0,-0.3) circle (1.5pt);
            \node[red,above] at (0,-0.6) {$P_2$}; 
    
            \filldraw[red] (2.5,0.9) circle (1.5pt);
            \filldraw[red] (2.5,0.6) circle (1.5pt);
            \filldraw[red] (2.5,0.3) circle (1.5pt);
            \node[red,below] at (2.5,0.6) {$P_1$}; 
            
            \filldraw[red] (0.95,0) circle (1.5pt);
            \filldraw[red] (1.25,0) circle (1.5pt);
            \filldraw[red] (1.55,0) circle (1.5pt);
            \node[red,left] at (1.25,0) {$Q$}; 
    
            \filldraw (0,-1.25) circle (2pt) node[below] {$(y,b)$};
            \filldraw (0,0) circle (2pt) node[below] {$(z,b)$};
            \filldraw (0,1.25) circle (2pt) node[below] {$(x,b)$};
    
            \filldraw (2.5,-1.25) circle (2pt) node[above] {$(y,a)$};
            \filldraw (2.5,0) circle (2pt) node[above] {$(z,a)$};
            \filldraw (2.5,1.25) circle (2pt) node[above] {$(x,a)$};
        \end{tikzpicture}
        \caption{$(x,a),(y,b)$-path $P$ containing $(z,a)$}
        \label{fig:product-case1}
        \end{figure}

        The case when $z\in\{x,y\}$ and $c\notin \{a,b\}$ is symmetric to the previous one and can be proved analogously.

        \item[Case 2.] If $z\notin\{x,y\}$ and $c\notin \{a,b\}$.

        Consider the paths $P_1$ and $P_2$ as defined in the previous case (that is, they are induced $x,z$- and $z,y$-paths, respectively, which only share the vertex $z$. In a similar way, by using Lemma~\ref{l:disjointpaths}, there exist two induced paths in $H$, $Q_1:a=a_0,\ldots,a_{j_1}=c$ and $Q_2:c=b_0,\ldots,b_{j_2}=b$ that only share the vertex $c$. 
        Now, consider the path $R:(x,a)=(v_0,a),\ldots,(v_{k_1},a)=(z,a)=(z,a_0),\ldots,(z,a_{j_1})=(z,c)=(u_0,c),\ldots,(u_{k_2},c)=(y,c)=(y,b_0),\ldots,(y,b_{j_2})=(y,b)$. (See also Fig.~\ref{fig:product-case2}.) It is clear that $R$ is an $(x,a),(y,b)$-path that contains vertex $(z,c)$. Each of the four sections of $R$, namely the $(x,a),(z,a)$-subpath of $R$, the $(z,a),(z,c)$-subpath of $R$, the $(z,c),(y,c)$-subpath of $R$ and $(y,c),(y,b)$-subpath of $R$ are induced. Since the concatenation of the paths $P_1$ and $P_2$ (as well as $Q_1$ and $Q_2$) is a path due to  Lemma~\ref{l:disjointpaths}, we infer that $R$ is induced.

        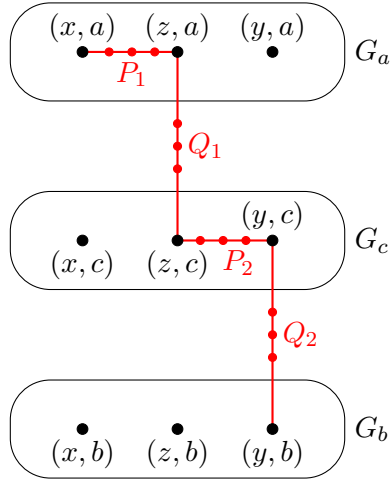
\begin{figure}[ht]
        \centering
        \begin{tikzpicture}[rotate=90]
            \draw[rounded corners=15pt] (-3.15,-2.2) rectangle (-1.85,2.2);
            \node[right] at (-2.5,-2.2) {$G_b$};
            \draw[rounded corners=15pt] (-0.65,-2.2) rectangle (0.65,2.2);
            \node[right] at (0,-2.2) {$G_c$};
            \draw[rounded corners=15pt] (1.85,-2.2) rectangle (3.15,2.2);
            \node[right] at (2.5,-2.2) {$G_a$};
    
            \draw[thick, red] (-2.5,-1.25) -- (0,-1.25) -- (0,0) -- (2.5,0) -- (2.5,1.25);
    
            \filldraw[red] (0,-0.9) circle (1.5pt);
            \filldraw[red] (0,-0.6) circle (1.5pt);
            \filldraw[red] (0,-0.3) circle (1.5pt);
            \node[red,below] at (0,-0.8) {$P_2$}; 
    
            \filldraw[red] (2.5,0.9) circle (1.5pt);
            \filldraw[red] (2.5,0.6) circle (1.5pt);
            \filldraw[red] (2.5,0.3) circle (1.5pt);
            \node[red,below] at (2.5,0.6) {$P_1$}; 
            
            \filldraw[red] (0.95,0) circle (1.5pt);
            \filldraw[red] (1.25,0) circle (1.5pt);
            \filldraw[red] (1.55,0) circle (1.5pt);
            \node[red,right] at (1.25,0) {$Q_1$}; 
            
            \filldraw[red] (-0.95,-1.25) circle (1.5pt);
            \filldraw[red] (-1.25,-1.25) circle (1.5pt);
            \filldraw[red] (-1.55,-1.25) circle (1.5pt);
            \node[red,right] at (-1.25,-1.25) {$Q_2$}; 
    
            \filldraw (-2.5,-1.25) circle (2pt) node[below] {$(y,b)$};
            \filldraw (-2.5,0) circle (2pt) node[below] {$(z,b)$};
            \filldraw (-2.5,1.25) circle (2pt) node[below] {$(x,b)$};
    
            \filldraw (0,-1.25) circle (2pt) node[above] {$(y,c)$};
            \filldraw (0,0) circle (2pt) node[below] {$(z,c)$};
            \filldraw (0,1.25) circle (2pt) node[below] {$(x,c)$};
    
            \filldraw (2.5,-1.25) circle (2pt) node[above] {$(y,a)$};
            \filldraw (2.5,0) circle (2pt) node[above] {$(z,a)$};
            \filldraw (2.5,1.25) circle (2pt) node[above] {$(x,a)$};
        \end{tikzpicture}
        \caption{$(x,a),(y,b)$-path containing $(z,c)$}
        \label{fig:product-case2}
        \end{figure}
    \end{description}
\end{proof}

 From Theorem~\ref{thm:product} we immediately infer that all hypercubes are strongly $2$-monophonic.
Now, note that also $K_n$ has the property from Lemma~\ref{l:disjointpaths} that for any three distinct vertices $x,y,z$, there exist an induced $x,y$-path $P$ and an induced $y,z$-path $Q$ such that the only vertex in both $P$ and $Q$ is $y$. Hence, letting $H=K_n$ in Theorem~\ref{thm:product} and using the same proof we can derive the following.

\begin{corollary}\label{c:productComplete}
    If $G$ is strongly 2-monophonic graph not isomorphic to $P_3$, then $G \Box K_n$ is a strongly 2-monophonic graph, for any $n\geq 1$.
\end{corollary}

Corollary~\ref{c:productComplete} implies that the sufficient condition in Theorem~\ref{thm:product} that both factor graphs are strongly $2$-monophonic is not necessary for the Cartesian product of two graphs to be strongly $2$-monophonic. Moreover, it is easy to see that the graphs $K_m\Box K_n$, where $m\ge n\ge 2$, are all strongly $2$-monophonic, yet whenever $n\ge 3$ none of the factors is in this class. This together with Corollary~\ref{c:productComplete} implies that non-complete Hamming graphs are strongly $2$-monophonic.

\begin{corollary}\label{cor:hypercubes-S2M}
A Hamming graph $H_{m_1,\ldots,m_k}$  is strongly $2$-monophonic except when $k=1$ and $m_1\ge 3$.
\end{corollary}

\subsection{Johnson graphs}

In this section we show that also Johnson graphs are strongly 2-monophonic. First, since for any $r$ the graph $J(n,r)$ is isomorphic to $J(n,n-r)$, we may assume without loss of generality that $2 \leq r \leq \lfloor \frac{n}{2} \rfloor$. We will use several basic properties of Johnson graphs. 
As mentioned earlier, Johnson graphs form a subfamily of generalized Johnson graphs, whose diameter and distance properties were studied in~\cite{agong2028diameter}.
First, for any $A,B \in V(J(n,r))$ it follows that $d(A,B)=r-|A \cap B|$ and consequently ${\rm diam}(J(n,r))=\min{\{r,n-r\}}$. Clearly, if $A=\{x_1,\ldots ,x_i,x_{i+1},\ldots , x_r\}$, $B=\{x_1,\ldots ,x_i,y_{i+1},\ldots ,y_r\}$ with $|A \cap B|= i$, then $A, A_{i+1},A_{i+2},\ldots , A_r=B$, is a shortest $A,B$-path if we define $A_i:=A$ and for $j \geq i+1$, $A_{j}:=(A_{j-1}\setminus \{x_{j}\}) \cup \{y_j\}$.

\begin{theorem}\label{thm:johnson-strong}
    If $r\geq 2$, then $J(n,r)$ is strongly 2-monophonic.
\end{theorem}
\begin{proof}
    Let $x$ and $y$ be two non-adjacent vertices of $J(n,r)$, this is, $|x\cap y|\leq r-2$. We shall prove that $S=\{x,y\}$ is a monophonic set of $J(n,r)$. That is, we will prove that every vertex $u\in V(G)$ lies on an induced $x,y$-path. 
    As this is clear for $x$ and $y$, let $u\in V(J(n,r))\setminus S$. Let $t=|u\cap x|$ and $s=|u\cap y|$. 
    
    Let us define $X=x\setminus(y\cup u)$, $Y=y\setminus(x\cup u)$, $U=u\setminus(x\cup y)$, $A=(x\cap y)\setminus u$, $B=(y\cap u)\setminus x$, $C=(x\cap u)\setminus y$ and $D=x\cap y\cap u$. Note that $|X\cup A|=|U\cup B|=r-t\geq 1$ and $|Y\cup A|=|U\cup C|=r-s \geq 1$.
    Let $X'=x\setminus u=A\cup X=\{x_1,\ldots,x_{r-t}\}$ and $Y'=y\setminus u=Y\cup A=\{y_1,\ldots,y_{r-s}\}$, where $x_i=y_{i+|Y|}\in A$ for $i\leq |A|$. Let $B'=u\setminus x=U\cup B=\{b_1,\ldots, b_{r-t}\}$ and $C'=u\setminus y=C\cup U=\{c_1,\ldots, c_{r-s}\}$, where $b_i=c_{i+|C|}\in U$ for $i\leq |U|$.
    
    For $i\in \{0\}\cup[r-t]$ and $j\in \{0\}\cup[r-s]$, let
    $$w_i=B'_{\leq i}\cup X'_{\geq r-t-i}\cup C\cup D,$$
    $$v_j=C'_{\geq j}\cup Y'_{\leq r-s-j}\cup B\cup D,$$
    Notice that the path $P:x=w_0,\ldots,w_{r-t}=u$ is a shortest $x,u$-path, since $d(x,u)=r-|x \cap u|=r-t$ and thus $P$ is also induced. Similarly, the path $Q:y=v_0,\ldots,v_{r-s}=u$ is an induced $y,u$-path. 
    We claim that $x=w_0,w_1,\ldots,w_{r-t}=u=v_{r-s},v_{r-s-1},\ldots,v_0=y$ is an induced $x,y$-path which clearly contains $u$.
    Let $w=w_i$ for some $i\in \{0\}\cup [r-t-1]$ and let $v=v_j$ for some $j\in \{0\}\cup [r-s-1]$. We shall show that $w\neq v$ and $w$ is not adjacent to $v$. It suffices to show that $|w\cap v|\leq r-2$, or equivalently, $|v\setminus w|\geq 2$.

    Note that $y_1\in v$ since $j<r-s$. Further, if $Y\neq \emptyset$, $y_1\in Y$ so $y_1\notin w$, and if $Y=\emptyset$, then $y_1=x_1\in A$ and $x_1\notin w$ if $i>0$. If $i=0$, then $B\subseteq v\setminus w$ and $|B|=|y\setminus x|\geq 2$.
    On the other hand, notice that $b_{r-t}\notin w$ since $i<r-t$. If $B\neq\emptyset$, then $b_{r-t}\in B\subseteq v$. If $B=\emptyset$, then $b_{r-t}\in U$ and $b_{r-t}=c_{r-s}\in v$ for $j\neq 0$. 
    So, unless $B=\emptyset$ and $j=0$ or $Y=\emptyset$ and $i=0$, we get $\{b_{r-t},y_1\}\subseteq v\setminus w$.
    In the former cases, we have $|v\setminus w|\geq |y\setminus x|\geq 2$.
    In either case, $|v\setminus w|\geq 2$.
\end{proof}

\section{Concluding remarks}

In this paper, we studied the monophonic number of Kneser graphs. We proved that $m(K(n,2))=3$ for any $n \geq 5$ and that $m(K(n,r))=2$ for any $r\geq 3$ and $n\geq 2r+1$. Moreover, we established a stronger result: in any Kneser graph $K(n,r)$ with $m(K(n,r))=2$, any pair of non-adjacent vertices forms a monophonic set. This motivated us to introduce the class of strongly 2-monophonic graphs, defined as graphs $G$ with $m(G)=2$ and the property that for every pair of non-adjacent vertices $x,y$, the set $\{x,y\}$ is a monophonic set. 

As mentioned in Section~\ref{ss:chordal}, adding or removing a universal vertex does not change whether a graph is strongly $2$-monophonic. Similarly, if $G$ is a strongly $2$-monophonic graph, and $u$ and $v$ are open twins in $G$ (that is, $N_G(u)=N_G(v)$), then by Proposition~\ref{prp:necessary}(iii), $u$ is a universal vertex in $G-v$, and so $G-\{u,v\}$ is also strongly $2$-monophonic. Therefore, when investigating which graphs are strongly 2-monophonic, we may restrict to the graphs without universal vertices and without any open twins. 
Any other strongly 2-monophonic graph can be obtained from such ``basic'' strongly 2-monophonic graphs by performing the operations of adding a universal vertex and adding an open twin to a universal vertex (the operations can be performed several times in any order).   
For instance, from $G=2K_1$, which is the smallest strongly $2$-monophonic graph, we can obtain by the described operations the graphs $K_n-M$, where $M$ is a matching that is removed from the complete graph of order $n$.

Results in the paper indicate that several families of graphs  with high symmetries (Kneser graphs, Johnson graph, cocktail-party graphs, Hamming graphs) are strongly $2$-monophonic, while the graphs with some kind of special vertices (such as chordal graphs) do not have this property. 
Hence, strongly 2-monophonic graphs may be of independent interest and may hide some intriguing qualities. This leads us to the following problem.

\begin{problem}
Characterize strongly 2-monophonic graphs.
\end{problem}

The above problem is interesting also from the computational aspect. As mentioned earlier determining $m(G)$ is NP-hard. However, this may not be the case for the studied class of graphs, and we ask the following.

\begin{problem}
    Determine the computational complexity of recognizing strongly 2-monophonic graphs.
\end{problem}

Possibly less ambitious but still interesting direction is to seek characterizations of strongly 2-monophonic graphs within specific graph classes. For example, we have given such results  for chordal graphs and dismantlable graphs. In particular, the following question arises. 

\begin{question}
    Which vertex-transitive graphs are strongly 2-monophonic? 
\end{question}

Consider generalized Johnson graph $J(6,4,2)$, and let $x,y$ be two nonadjacent vertices. Without loss of generality, $x=\{1,2,3,4\}$ and $y=\{1,2,3,5\}$. It is not hard to see that $\{1,2,3,6\}$ is not in the monophonic interval between $x$ and $y$, whereas $S=\{x,y,\{1,2,3,6\}\}$ is a monophonic set. Hence, $m(J(6,4,2))=3$.
If $n\geq 7$, then for each pair $x,y$ of nonadjacent vertices, the set $\{x,y\}$ is a monophonic set and $J(n,4,2)$ turns out to be strongly 2-monophonic.

\begin{problem}
    Determine all generalized Johnson graphs which are (strongly) 2-monophonic.
\end{problem}
\begin{question}
    Are all 2-monophonic generalized Johnson graphs strongly 2-monophonic?
\end{question}

Finally, our investigation revealed that graphs with monophonic number 2
are not yet well understood. Thus, a broader study of 2-monophonic graphs seems both natural and promising.

\section*{Acknowledgments}
B.B. and T.D. acknowledge the financial support of the Slovenian Research and Innovation Agency (research core funding No.\ P1-0297, projects N1-0285, N1-0431).
M.G.C. was partially supported by the Argentinian National Agency for the Promotion of Research, Technological Development and Innovation (project PICT-2020-03032), the Argentinian National Council for Scientific and Technical Research (project PIP CONICET 1900) and the National University of Rosario (project PID 80020210300068UR).
M.G.C. gratefully acknowledges that this work was carried out during a research visit to the University of Maribor as part of the UM Visiting Researcher Programme 2025.

\bibliographystyle{apalike}
\bibliography{biblio}

\end{document}